\documentclass{amsart}

\usepackage{amssymb,amsmath,amsthm}
\usepackage{graphics}
\usepackage[arrow, matrix, curve]{xy}
\usepackage{bbm}

\theoremstyle{definition}
\newtheorem{definition}{Definition}[section]
\newtheorem{example}[definition]{Example}
\theoremstyle{plain}
\newtheorem{theorem}{Theorem}

\newtheorem*{ack}{Acknowledgment}
\newtheorem{lemma}[definition]{Lemma}
\newtheorem{corollary}[definition]{Corollary}
\newtheorem{proposition}[definition]{Proposition}

\theoremstyle{remark}

\theoremstyle{dotless}

\newcommand{\N}{\mathbb{N}}
\newcommand{\Q}{\mathbb{Q}}
\newcommand{\R}{\mathbb{R}}

\newcommand{\Rop}{[0,\infty]}

\newcommand{\eps}{\varepsilon}

\renewcommand{\rho}{\varrho}
\newcommand{\pphi}{\hat\phi}
\newcommand{\ppi}{\hat\pi}
\newcommand{\ppsi}{\hat\psi}

\newcommand{\ovlc}{\overline c}
\newcommand{\unlc}{\underline c}
\newcommand{\me}{^{-1}}
\newcommand{\push}{{}_{\#}}

\renewcommand{\subset}{\subseteq}
\newcommand{\Assumptions}{Assume that $X,Y$ are Polish spaces equipped with Borel probability
 measures $\mu,\nu$, that $c:X\times Y\to [0,\infty]$ is Borel measurable
 and $\mu\otimes\nu$-a.e.\ finite and that there exists a finite transport plan.}

\title{Duality for Borel measurable cost functions}
 \author{Mathias Beiglb\"ock}

 \author{Walter Schachermayer}
\address{Institut f\"ur Diskrete Mathematik und Geometrie, Technische Universit\"at Wien\endgraf
Wiedner Hauptstra\ss e 8-10/104\\ 1040 Wien, Austria}
\email{mathias.beiglboeck@tuwien.ac.at}
\address{Institut f\"ur Wirtschaftsmathematik, Technische Universit\"at Wien\endgraf
Wiedner Hauptstra\ss e 8-10/105, 1040 Wien, Austria}
\email{walter.schachermayer@tuwien.ac.at}
\thanks{The first author gratefully acknowledges financial support from the Austrian Science
Fund (FWF) under grant S9612. The second author gratefully
acknowledges financial support from the Austrian Science Fund (FWF)
under grant P19456, from the Vienna Science and Technology Fund
(WWTF) under grant MA13 and from the Christian Doppler Research
Association (CDG)}

\begin{document}

\begin{abstract}
We consider the Monge-Kantorovich transport problem in an abstract
measure theoretic setting. 
Our main result states that duality holds if $c:X\times Y\to
[0,\infty)$ is an arbitrary Borel measurable
cost function on the product of Polish spaces $X,Y$.
In the course of the proof we show how to relate a non - optimal
transport plan to the optimal transport costs via a ``subsidy''
function and how to identify the dual optimizer. We also provide some examples showing the limitations  of
the duality relations.
\end{abstract}
\maketitle
\section{Introduction} We consider the \emph{Monge-Kantorovich transport
problem} for Borel probability measures  $\mu,\nu$ on  Polish spaces
$X,Y$. See \cite{RaRu98,Vill03,Vill05} for  a general account of the
theory of optimal transportation.
 The \emph{cost function}  $c:X\times Y\to [0,\infty]$ is
assumed to be Borel measurable. $\Pi(\mu,\nu)$ is the set of all
\emph{transport plans}, that is,  Borel probability measures on
$X\times Y$ which have $X$-marginal $\mu$ and $Y$-marginal $\nu$.
The \emph{transport costs} associated to a transport plan $\pi$  are
given by
\begin{equation}I_c[\pi]=\int_{X\times Y}c\,d\pi\end{equation} and we
say that $\pi$ is a \emph{finite} transport plan if
$I_c[\pi]<\infty$. The infimum over all possible transport costs
$I_c[\pi],\pi\in\Pi(\mu,\nu)$ will be denoted by $I_c$.
 We
define $\Phi(\mu,\nu)$ as the set of pairs $(\phi,\psi)$ of
integrable functions $\phi:X\to[-\infty,\infty)$ and
$\psi:Y\to[-\infty,\infty)$ which satisfy $\phi(x)+\psi(y)\leq
c(x,y)$ for all $(x,y)\in X\times Y$. The dual part of the
Monge-Kantorovich problem then consists in maximizing
\begin{equation}\label{SimpleJ}
J(\phi,\psi)=\int_{X}\phi\,d\mu+\int_{Y}\psi\,d\nu
\end{equation}
for $(\phi,\psi)\in\Phi(\mu,\nu)$. Monge-Kantorovich duality asserts
that $\inf \{I_c[\pi]:\pi\in
\Pi(\mu,\nu)\}=\sup\{J(\phi,\psi):(\phi,\psi)\in\Phi(\mu,\nu)\}$.
For example, if $X$ and $Y$ consist of $N$ points, each having
measure $1/N$, duality holds true, as this reduces to an elementary
linear programming problem. In the literature duality has been
established under various conditions, see for instance \cite[p
98f]{Vill05} for a short overview. In particular it is known that
duality holds if $c$ is lower semi-continuous (see \cite[Theorem
2.2]{Kell84} or \cite[Theorem 5.10]{Vill05}) or just Borel
measurable and bounded by the sum of two integrable functions
(\cite[Theorem 2.14]{Kell84}).  In \cite{RaRu95,RaRu96} the problem
is investigated beyond the realm of Polish spaces and it is
characterized for which spaces duality holds for all bounded
measurable cost functions.

Our main result is that Monge-Kantorovich duality holds in the case
of a finite but not necessarily bounded Borel measurable cost
function.
\begin{theorem}\label{BorelDuality}
\Assumptions\ Then
\begin{align}\label{DualityEquation}
I_c=\inf \{I_c[\pi]:\pi\in
\Pi(\mu,\nu)\}=\sup\{J(\phi,\psi):(\phi,\psi)\in\Phi(\mu,\nu)\}.
\end{align}
\end{theorem}
In contrast to the situation where $c$ is lower semi-continuous,
duality does not hold in general in the Borel setting if $c$ attains
the value $\infty$ on a large set, see Example \ref{ZeroOneInfty}.

\subsection{Existence of dual maximizers}

\medskip

In general it is not possible to find dual maximizers $\pphi$ and
$\ppsi$ for (\ref{SimpleJ}) which are \emph{integrable}, even if the
cost function is the squared distance on $\mathbb R$ (see Example
\ref{NoIntegrableMaximizers2} below). However it is possible to find
dual maximizers in a weaker sense for all  Borel measurable cost
functions which are $\mu\otimes\nu$-a.e.\ finite. Assume that $\pi$
is a finite transport plan and that $\phi,\psi$ are arbitrary
functions satisfying $\phi(x)+\psi(y)\leq c(x,y)$. While
$\phi(x)+\psi(y)$ is not necessarily integrable with respect to
$\mu\otimes \nu$, it can be integrated with respect to $\pi$
(possibly assuming the value $-\infty$). Thus we may well define
\begin{align}\label{ExtendedDefinition}
J(\phi,\psi)=\int_{X\times Y}\big[ \phi(x)+\psi(y)\big]\,d\pi(x,y).
\end{align}
The subsequent lemma shows that the notation $J(\phi,\psi)$ is justified in the sense  that this
definition does not depend on the particular choice of $\pi$.

 \begin{lemma}\label{IndependenceOfPi} Let $X,Y$ be Polish spaces equipped
 with Borel probability measures $\mu,\nu$ and let $c:X\times Y\to[0,\infty]$ be a
  Borel measurable cost  function.
Let $\pi,\tilde\pi \in \Pi(\mu,\nu)$ be finite transport plans and assume
 that $\phi:X\to [-\infty,\infty),\psi:Y\to [-\infty,\infty)$ are such that
  $\phi(x)+\psi(y)\leq c(x,y)$  
  holds $\pi$-almost surely as well as $\tilde \pi$-almost surely.
 Then \begin{align}
 \int_{X\times Y} \big[\phi(x)+\psi(y)\big]\,d\pi(x,y)=
 \int_{X\times Y}\big[ \phi(x)+\psi(y)\big]\,d\tilde\pi(x,y).\end{align}
 Moreover  there is a sequence $(\phi_n,\psi_n)_{n\geq 1} $
 of bounded functions in $\Phi(\mu,\nu)$ such that $\lim_{n\to\infty}
  J(\phi_n,\psi_n)=\int_{X\times Y} \big[\phi(x)+\psi(y)\big]\,d\pi(x,y)$.
\end{lemma}

Using Lemma \ref{IndependenceOfPi}, Theorem \ref{BorelDuality}
follows from Theorem \ref{DualMaximizers} below. Theorem
\ref{DualMaximizers} is stronger in the sense that it guarantees
that the supremum on the right side of (\ref{DualityEquation}) is in
fact a maximum if $J(\phi,\psi)$ is defined as in
(\ref{ExtendedDefinition}).

\begin{theorem}\label{DualMaximizers}
\Assumptions\ Then  there exist Borel measurable
\emph{dual maximizers} $\pphi,\ppsi$, i.e.\ functions $\pphi:X\to [-\infty,\infty),\ppsi:Y\to
[-\infty,\infty)$ satisfying $\pphi(x)+\ppsi(y)\leq c(x,y)$ for all
$(x,y)\in X\times Y$ such that
\begin{align}\label{DualityEquation2}
I_c= \inf \{I_c[\pi]:\pi\in \Pi(\mu,\nu)\}=J\big(\pphi,\ppsi\big).
\end{align}
\end{theorem}
We point out that the hypothesis that $c$ is $\mu\otimes\nu$-a.e.\
finite is crucial  for the existence of dual optimizers (see Example
\ref{ExAmPra}). Even in the case of a continuous cost function
$c:X\times Y\to [0, \infty]$, it is in general not possible to find
dual maximizers  although duality holds in this case (see Example
\ref{ContinuousNoMaximizers}).

\medskip

In \cite[Theorem 3.2]{AmPr03} it is proved that in the case of a finite
lower semi-continuous cost function there exist integrable dual
maximizers $(\pphi,\ppsi)\in\Phi(\mu,\nu)$, provided that
\begin{align}\label{IntEins}
\mu\Big( \Big\{ x\in X:\int_Y c(x,y)\,d\nu(y)<\infty\Big\}\Big)>0,\\
\nu\Big( \Big\{ y\in Y:\int_X
c(x,y)\,d\mu(x)<\infty\Big\}\Big)>0.\label{IntZwei}
\end{align}
Their argument yields that functions
$\pphi,\ppsi$ which are dual maximizers in the sense of Theorem \ref{DualMaximizers} are  $\mu$- resp.\ $\nu$-integrable whenever
(\ref{IntEins}) and (\ref{IntZwei}) are satisfied.

\subsection{Dual maximizers and strong $c$-cyclical monotonicity}

\medskip

Theorem \ref{DualMaximizers} is also connected with the notion of
strong $c$-cyclical monotonicity introduced  in \cite{ScTe08}. A
transport plan $\pi$ is \emph{strongly $c$-cyclically
monotone}\footnote{This notion is called strong $c$-monotonicity in
\cite{ScTe08,BGMS08}. We find it, however, more consistent with
previous notations in the literature to call it strong $c$-cyclical
monotonicity.} if there exist Borel measurable functions $\phi:X\to
[-\infty,\infty), \psi:Y\to [-\infty,\infty)$ such that
$\phi(x)+\psi(y)\leq c(x,y)$ for all $(x,y)\in X\times Y$ and
$\phi(x)+\psi(y)=c(x,y)$ for $\pi$-almost all $(x,y)\in X\times Y$.
In \cite{BGMS08} it is proved that in the case of a finite cost
function, a transport plan is strongly $c$-cyclically monotone if
and only if it is optimal. We want to point out that this is also a
consequence of Theorem \ref{DualMaximizers}:

\begin{corollary}\label{ConnectionWithStrongCmon} \Assumptions\
\begin{enumerate}
\item Let $\pi $ be a finite transport plan and assume that $\phi:X\to [-\infty,\infty), \psi:Y\to [-\infty,\infty)$ witness that $\pi$ is strongly $c$-cyclically monotone. Then $J(\phi,\psi)=I_c[\pi]$, thus $\pi$ is an optimal transport plan and $\phi,\psi$ are dual maximizers.
\item Assume that $\ppi$ is an optimal transport plan, i.e.\ $I_c[\ppi]=I_c$. Then $\ppi$ is strongly $c$-cyclically monotone. In fact, this is witnessed by every  pair $(\pphi,\ppsi)$ of dual maximizers.
\end{enumerate}
\end{corollary}
\begin{proof}
Given a transport plan $\pi$ and functions $\phi,\psi$ witnessing
that $\pi$ is strongly $c$-cyclically monotone, we have \begin{align} I_c[\pi]= \int_{X\times Y} c(x,y) \,d\pi(x,y)
=\int_{X\times Y}\big[ \phi(x)+\psi(y)\big] \,d\pi(x,y)=J(\phi,\psi)\leq I_c. \end{align} Thus
$I_c[\pi]=I_c$, hence $\pi$ is an optimal transport and $\phi,\psi $
are dual maximizers in the sense of Theorem \ref{DualMaximizers}.

\medskip

Conversely assume that $\ppi$ is an optimal transport plan and that
$\pphi,\ppsi$ are dual maximizers.
Then
\begin{align}
0=I_c[\ppi]-J(\pphi,\ppsi)=\int_{X\times Y}
c(x,y)-\big[\pphi(x)+\ppsi(y)\big]\,d\ppi(x,y).
\end{align}
Hence $c(x,y)- \big[\pphi(x)+\ppsi(y)\big]=0$ for $\ppi$-almost all $(x,y)\in
X\times Y$. Thus, $\pphi,\ppsi$ witness that $\ppi$ is strongly
$c$-cyclically monotone.
\end{proof}

\subsection{Continuity of $I_c$.}

The fact that Monge-Kantorovich duality holds for not necessarily finite lower semi-continuous cost functions is related to a certain continuity property of the mapping  $c\mapsto I_c$ which is always satisfied in the lower semi-continuous setting.

If the cost function  $c:X\times Y\to [0,\infty]$ is lower
semi-continuous, there exists a sequence $(c_n)_{n\geq 1}$ of
bounded continuous functions such that $c_n\uparrow c$.  For each
such sequence we have $I_{c_n}\uparrow I_c$. This can easily be
derived from the fact that $\Pi(\mu,\nu)$ is weakly
compact\footnote{This is a consequence of Prokhorov's Theorem, see
\cite[p 56]{Vill05}.}: Pick for each $n\geq 1$ a transport plan
$\pi_n$ such that $I_{c_n}[\pi_n]\leq I_{c_n}+1/n$. By passing to a
subsequence if necessary we may assume that $(\pi_n)_{n\geq 1}$
converges weakly to some transport plan $\pi\in\Pi(\mu,\nu)$. Then
\begin{align}
I_c\leq  I_c[\pi] =\lim_{m\to \infty}\int c_m \, d\pi&=\lim_{m\to \infty}\left(\lim_{n\to \infty}\int c_m \, d\pi_n\right)\\
&\leq \lim_{m\to \infty} \left(\lim_{n\to \infty}\int c_n \, d\pi_n\right)=\lim_{n\to\infty}I_{c_n}.
\end{align}
Since $I_{c_n}\leq I_c$ it follows in fact that $I_{c_n}\uparrow I_c$.  Observe  that $I_c[\pi]=I_c$, i.e.\ $\pi$ is a primal optimizer.

A direct consequence of this simple continuity result is that as soon as we have shown the relatively easy result that duality holds for bounded continuous functions, it already follows for an arbitrary lower semi-continuous function $c:X\times Y\to [0,\infty]$. To see this, pick a sequence of bounded continuous functions  $c_n:X\times Y\to [0,\infty)$ and for each $n\geq 1$ a pair of integrable  functions $(\phi_n,\psi_n) $  such that
 $J(\phi_n,\psi_n)\geq I_{c_n}-1/n$ and $\phi_n(x)+\psi_n(y) \leq c_n(x,y)\ \big(\leq c(x,y)\big)$ on $X\times Y$. Then
\begin{align}
\sup_{n\geq 1} \{J(\phi_n,\psi_n\}\geq \lim_{n\to \infty} (I_{c_n}-1/n)  =I_c,
\end{align}
thus duality holds.

A similar continuity property holds in the case of a finite measurable cost function. However we do not know how continuity in this sense  can be shown directly, instead we achieve it as a consequence of Theorem \ref{BorelDuality}.
\begin{corollary}\label{IcContinuous}
\Assumptions\ Then $I_{c\wedge n}\uparrow I_c$.
\end{corollary}
\begin{proof}
Given $\eps>0$ there exist bounded functions $\phi:X\to \R, \psi:Y\to \R$ such that $J(\phi,\psi)>I_c-\eps$  by Theorem \ref{BorelDuality} and Lemma \ref{IndependenceOfPi}. For all large enough $n$, we have $\phi(x)+\psi(y)\leq (c\wedge n)(x,y)$ on $X\times Y$. Thus $I_{c_n}\geq J(\phi,\psi)\geq I_c-\eps.$
\end{proof}

\subsection{Cost functions with negative values}
For notational convenience we have chosen to consider only
non-negative cost functions, but this restriction is somewhat
stronger than necessary.  Theorem \ref{BorelDuality} does remain
valid (and  in  fact so do our other results) in the setting of a
cost function $c:X\times Y\to [-\infty,\infty]$, provided that $c$
is $\mu\otimes \nu$-a.e.\ finite and that there exist integrable
functions $a:X\to [-\infty,\infty), b:Y\to [-\infty,\infty)$  such
that
\begin{align}
a(x)+b(y)\leq c(x,y)
\end{align}
for all $(x,y)\in X\times Y$. This is an immediate consequence of
Theorem \ref{BorelDuality} applied  to the cost function
$c(x,y)-a(x)-b(y)\footnote{Throughout this paper we use the
convention $\infty-\infty=\infty$.}$.

\section{$c$-cyclical monotonicity}
A transport plan  $\pi$  is \emph{$c$-cyclically monotone} if it is
concentrated on a Borel set $\Gamma\subset X\times Y$ which is $c$-cyclically monotone in the sense  that
\begin{equation}\label{QuantifiedMonotonicityEquation}
 \sum_{i=1}^n c(x_{i+1},y_i)-c(x_{i},y_i)\geq 0.
 \end{equation} for  all $(x_1, y_1),\ldots,(x_n,y_n)\in \Gamma$. (Here we let
 $x_{n+1}=x_1$.) Heuristically, $\pi$ is $c$-cyclically monotone if it cannot
 be enhanced by means of cyclical rerouting. Hence it is intuitively appealing (and obvious in the finite setting) that optimal transport plans are always $c$-cyclically monotone. In fact it can be shown that in the case of a Borel measurable cost function $c:X\times Y\to [0,\infty]$ every optimal transport plan is $c$-cyclically monotone,
 and that the two notions  are equivalent if $c$ is finitely valued (\cite[Theorem 1]{BGMS08}).
 (This equivalence is not true in general, as shown by a beautiful counterexample due
 to Ambrosio and Pratelli, \cite[Example 3.5]{AmPr03}. The connection between optimality and $c$-cyclical monotonicity was also studied in \cite{GaMc96,AmPr03,Prat08,ScTe08}.)

The concept of $c$-cyclical monotonicity is crucial for the
Monge-Kantorovich duality. We shall review its connection with
strong $c$-cyclical monotonicity. As indicated by the name, it is
almost obvious to see that strong $c$-cyclical monotonicity implies
$c$-cyclical monotonicity: Assume that $\phi, \psi$ witness that
$\pi$ is strongly $c$-cyclically monotone. Then $\pi$ is
concentrated on the set $\Gamma=\{(x,y):c(x,y)=\phi(x)+\psi(y)\}$
and
\begin{align}\label{StrongCmonToCmon}
& \sum_{i=1}^n c(x_{i+1},y_i)-c(x_{i},y_i)=\\
& \sum_{i=1}^n c(x_{i+1},y_i)-\big[\phi(x_{i})+\psi(y_i)\big]\geq\\
& \sum_{i=1}^n \big[\phi(x_{i+1})+\psi(y_i)\big]-\big[\phi(x_{i})+\psi(y_i)\big]=0
\end{align}
for  all $(x_1, y_1),\ldots,(x_n,y_n)\in \Gamma$. Less trivially, the subsequent proposition asserts that the strong version can be deduced from the usual one,  provided that $c$ is finitely valued. (This is in general not the case if $c$ attains $\infty$ on a large set, see Example \ref{ExAmPra} below.)

\begin{proposition}\label{CmonToStrongCmon} \Assumptions\ Then every $c$-cyclically monotone transport plan $\pi$ is strongly $c$-cyclically monotone.
\end{proposition}
  Proposition \ref{CmonToStrongCmon}  can be proved using a well known
  construction given  in  \cite{R96}, see also \cite{Rock66}, \cite[Chapter 2]{Vill03} and \cite[Theorem 3.2]{AmPr03}. Assume for notational convenience that  
 $\pi$ is concentrated on a $c$-cyclically monotone set $\Gamma$ which satisfies $p_X[\Gamma]=X$ and $p_Y[\Gamma]=Y$ and that $c$ is finite on $X\times Y$. Then the definition  \begin{align}\label{PhiDefi}
\phi(x):=&\inf \!\left\{\sum_{i=1}^n\big[ c(x_{i+1},y_i)\!-\!c(x_i,y_i)\big]: (x_1,y_1),\ldots,(x_n,y_n)\in \Gamma, x_{n+1}= x\right\}\!\!\\
\psi(y):=& \inf_{x\in X} c(x,y)-\phi(x),
\end{align} where $x_1\in X$ is an arbitrary fixed point, yields functions witnessing that $\pi$ is strongly $c$-cyclically monotone.
Strictly speaking, it  might be necessary to alter $\phi$ and $\psi$ on appropriately chosen null sets to ensure that they are Borel measurable functions, but these are merely technical obstacles which we will not discuss at this point.
Instead we shall below  derive  Proposition  \ref{CmonToStrongCmon} rigorously from the more general result in Proposition \ref{SandwichTheorem}.
 
 \medskip
The subsequent statement summarizes how $c$-cyclical monotonicity connects to the other concepts discussed so far.
\begin{proposition}
\Assumptions\ Let $\pphi, \ppsi$ be dual maximizers in the sense of (\ref{ExtendedDefinition}). Then the set $\Gamma:=\{(x,y):\pphi(x)+\ppsi(y)=c(x,y)\}$ is $c$-cyclically monotone.

For any finite transport plan $\pi$ the following conditions are equivalent.
\begin{enumerate}
\item[(a)] $\pi$ is concentrated on $\Gamma$.
\item[(b)] $\pi$ is $c$-cyclically monotone.
\item[(c)] $\pi$ is strongly $c$-cyclically monotone.
\item[(d)] $\pi$ is optimal.
\end{enumerate}
\end{proposition}
\begin{proof}
To see that $\Gamma$ is $c$-cyclically monotone, argue as in
(\ref{StrongCmonToCmon}). As regards the equivalence of (a) to (d),
(a) trivially implies (b). We have seen above that (b) and (c) are
equivalent and (c) and (d) are equivalent by Corollary
\ref{ConnectionWithStrongCmon}.

If $\pi$ is an optimal transport plan, $\pphi,\ppsi$ witness that $\pi$ is strongly $c$-cyclically monotone by Corollary \ref{ConnectionWithStrongCmon} (2) such that $\pphi(x)+\ppsi(y)=c(x,y)$ holds  for $\pi$-almost all $(x,y)\in X\times Y$.  Hence (d) implies (a).
\end{proof}

\section{Transports with subsidies}

In the last section it was described how dual maximizers can be constructed starting from a $c$-cyclically monotone transport plan.
However, in the absence of lower semi-continuity of $c$, there is no reason why one should find a transport plan which is supported by a $c$-cyclically monotone set, even in very regular situations as is shown by the subsequent easy example.

\begin{example} Let $(X,\mu)=(Y,\nu)$ equal the unit interval $[0,1]$ equipped with Lebesgue measure.
Define the cost of moving $x$ to $y$ by $c(x,y)=(x-y)^2$ for $x\neq
  y$ and let $c(x,x)=1$. Clearly it is possible find transport plans
  with arbitrarily small costs, but the infimum $0$ is not
  attained. Thus there exists no
  optimal and hence no $c$-cyclically monotone transport plan.

  Observe that  dual maximizers exist; just  set $\pphi\equiv\ppsi \equiv0$.
\end{example}

Our attempt to overcome this difficulty is to introduce a certain
\emph{subsidy function}. To explain this notion we take up the
anecdotal interpretation (see \cite[Chapter 3]{Vill05}) where
$(X,\mu)$ models the Parisien bakeries (i.e., croissant sellers) and
$(Y, \nu)$ the Parisian caf\'es (i.e., croissant buyers) and
$c(x,y)$ denotes the transport cost from bakery $x$ to cafe $y$. To
avoid technicalities we suppose that there are only finitely many
caf\'es and bakeries in Paris and that $c(x,y)$ is finitely valued.
(We are not sure to which degree this assumption corresponds to
reality.) Suppose that $\pi$ is the traditional way how the
croissants are transported from the bakeries to the caf\'es and
that, for whatever reason, the Parisian authorities want to maintain
this transport way also in the future. If the difference $\alpha$
between the present costs $I_c[\pi]$  and the cheapest possible
transport costs $I_c$ is strictly positive and the authorities do
not intervene, they should expect that market forces will sooner or
later cause the transport scheme  to switch from $\pi$ to some other
$\tilde \pi$ with lower total transport cost. Therefore they might
try to introduce a subsidy system, where the transport of each
croissant from $x$ to $y$ is subsidized by an amount  $f(x,y)\in
[0,\infty)$.

The aim of the Parisian authorities is to design the system $f(x,y)$ of subsidies in such a way that the daily total subsidies $F=\int_{X\times Y} f\,d\pi$ effectively paid are minimized under the constraint that the bakers and cafetiers have no rational incentive to change the traditional transport system $\pi$ by cyclically rerouting their ways of transportation.

A moment's reflection reveals that  a lower bound for the cost of
subsidy is given by
\begin{align}\label{LBS}\tag{LB}F=\int_{X\times Y} f\,d\pi\geq\alpha=I_c[\pi]-I_c,\end{align} and it will turn out  that this lower
bound is attained which should not be very surprising. In fact,
there are (at least) two versions of the ``no incentive to change''
constraint:
\begin{align}\label{S1}\tag{W1}
\sum_{i=1}^n c(x_{i+1},y_i)-\big(c(x_i,y_i)-f(x_i,y_i)\big) \geq 0
\end{align}
for all $(x_1,y_1),\ldots,(x_n,y_n)$ in the support of $\pi$ or
\begin{align}\label{S2}\tag{S1}
\sum_{i=1}^n \big(c(x_{i+1},y_i)-f(x_{i+1},y_i)\big)-\big(c(x_i,y_i)-f(x_i,y_i)\big) \geq 0
\end{align}
for all $(x_1,y_1),\ldots,(x_n,y_n)$ in the support of $\pi$.

An interpretation of the two  requirements goes as follows. In
(\ref{S1}) the authorities make a ``take it or leave it'' proposal
to the bakers and cafetiers: if you stick to the transport system
$\pi$ we pay the subsidies $f$, if not we pay nothing. Hence
comparing the transport costs from $x_i$ to $y_i$ with the ones from
$x_{i+1}$ to $y_{i}$, we have to make sure that for every collection
$(x_1,y_1),\ldots, (x_n,y_n)$ in the support of $\pi$  the
subsidized  costs $\sum_{i=1}^n c(x_i,y_i)-f(x_i,y_i)$ are less than
or equal to the non-subsidized costs after rerouting $\sum_{i=1}^n
c(x_{i+1},y_i)$, which amounts to (\ref{S1}).

In the interpretation of (\ref{S2}) the Parisian authorities behave
in a less authoritarian way: they promise to pay the subsidies
$f(x,y)$ independently of whether the bakers and cafetiers are
obedient or not, which amounts to the constraint (\ref{S2}).

In fact, and this seems somewhat surprising, the validity of (\ref{S1}) implies that there exists a function $f$ such that $F=\int_{X\times Y} f\, d\pi=\alpha$ and such that the subsequent constraint (\ref{S3}) which is yet stronger than (\ref{S2}) is satisfied:
\begin{align}\label{S3}\tag{S2}
\sum_{i=1}^n \big(c(x_{i+1},y_i)- f(x_{i+1},y_i)\big)-\big(c(x_i,y_i)-f(x_i,y_i)\big) \geq 0
\end{align}
\emph{for all} $(x_1,y_1),\ldots,(x_n,y_n)\in X\times Y$ (not necessarily being in the support of $\pi$).

Replacing $(x_1,y_1),\ldots,(x_n,y_n)$ by $(x_n,y_{n-1}),(x_{n-1},
y_{n-2}),\ldots,(x_2,y_1), (x_1,y_n)$ one verifies that one must
have \emph{equality} in (\ref{S3}) for all
$(x_1,y_1),\ldots,(x_n,y_n)\in X\times Y$. This may be interpreted
as follows: there is a subsidy function $f(x,y)$ with total subsidy
payment $F=\int_{X\times Y} f\, d\pi =\alpha$ and such that, for the
subsidized transport cost we have that $\int_{X\times Y}(c-f)\,
d\tilde\pi$ is equal for \emph{any} transport plan $\tilde \pi\in
\Pi(\mu,\nu)$ and if (\ref{LBS}) holds true, one easily verifies
that this value must equal $I_c$. In particular, the bakers and
cafetiers have no incentive to change $\pi$ as they are, in fact,
indifferent between all the possible transports $\tilde \pi\in
\Pi(\mu,\nu)$, if their goal is to minimize the total subsidized
transport costs.

In order to prove the existence of a subsidy system  $f(x,y)$
satisfying (\ref{S3}) and (\ref{LBS}) we use the following constraint
which strengthens (\ref{S1}) in a similar way as (\ref{S3})
strengthens (\ref{S2}):
\begin{align}\label{S5}\tag{W2}
\sum_{i=1}^n c(x_{i+1},y_i)-\big(c(x_i,y_i)-f(x_i,y_i)\big) \geq 0
\end{align}
for all $(x_1,y_1),\ldots,(x_n,y_n)\in X\times Y$.

We  refrain from giving an intuitive interpretation of (\ref{S5}).
 Rather we try to indicate on
an intuitive level why all four versions of the constraint are
equivalent when minimizing $F=\int_{X\times Y} f\, d\pi$.  More
precisely, suppose that there is a function $f(x,y)$ such that
$\int_{X\times Y} f \, d\pi=\alpha$  satisfying (the weakest form
of) constraint (\ref{S1}) and let us show that there exists a
function $\tilde f$ with $\int \tilde f \,d\tilde \pi=I_c[\tilde
\pi]-I_c$ for any transport plan $\tilde \pi$. Hence $\tilde f$
satisfies the (strongest form of) constraint (\ref{S3}).

To pass from (\ref{S1}) to (\ref{S5}) is a cheap shot:
  observe that (\ref{S1}) as well as $\int _{X\times Y} f\, d\pi$ only
  pertain to values of $f$ on the support of $\pi$.
  Hence we may alter $f$ outside the support of $\pi$ to be
  $\infty$ (in the case of finite $X$ and $Y$  we clearly may replace
  $\infty$ by a sufficiently large real number), which then trivially
  satisfies (\ref{S5}). To alleviate notation we still denote by $f$ the
   function satisfying (\ref{S5}). To pass from (\ref{S5}) to (\ref{S3})
   we observe the subsequent ``sandwich'' type result which seems
   interesting in its own right.

Denoting $\ovlc(x,y)=c(x,y)$ and
$\unlc(x,y):= c(x,y)-f(x,y)$ inequality (\ref{S5}) may be written as
\begin{align}\label{S7}\tag{W3}
\sum_{i=1}^n \ovlc(x_{i+1},y_i)-\unlc(x_i,y_i) \geq 0
\end{align}
for all $(x_1,y_1),\ldots,(x_n,y_n)\in X\times Y$.

We shall show that (\ref{S7}) implies (at least in our present setting of finite spaces $X$ and $Y$  and a finite cost function $c=\ovlc$) that we may find functions $\phi(x),\psi(y)$ such that
\begin{align}\label{S7a}\tag{W3a}
\ovlc(x,y)\geq \phi(x)+\psi(y)\geq \unlc(x,y),
\end{align}
for all $(x,y)\in X$.

\medskip

 To motivate why this result should indeed be considered as  a sandwhich theorem note the easy fact that a function $c(x,y)$ on $X\times Y $ may be written as
\begin{align}\label{S8a}
c(x,y)=\phi(x)+\psi(y)
\end{align}
if and only if it satisfies
\begin{align}\label{S8}
\sum_{i=1}^n c(x_{i+1},y_i)-c(x_i,y_i) = 0
\end{align}
for all $(x_1,y_1),\ldots,(x_n,y_n)\in X\times Y$. The problem under
which conditions a function defined on the product of two sets can
be decomposed as the sum of two univariate functions is studied in
detail in \cite{BoLe92}.
\medskip

A precise version of our sandwich theorem under the assumption that $\ovlc,\unlc$ are Borel measurable and $\ovlc$ is $\mu\otimes \nu$-a.e.\ finite  is given in Proposition \ref{SandwichTheorem}.
In our present situation it guarantees the existence of functions  $\phi, \psi$ satisfying
\begin{align}\label{W3b}\tag{W3b}
c(x,y)\geq \phi(x)+\psi(y)\geq c(x,y)-f(x,y).
\end{align} for all $(x,y)\in X\times Y$. This allows us to define
\begin{align}
\tilde f(x,y):= c(x,y)-\big(\phi(x)+\psi(y)\big).
\end{align}
The lower bound of (\ref{W3b}) implies that $f(x,y)\geq c(x,y)-\big( \phi(x)+\psi(y)\big)=\tilde f(x,y)$, in particular the desired bound for the total cost of subsidy $\int_{X\times Y} \tilde f\, d\pi \leq  \int_{X\times Y} f\, d\pi=\alpha$ holds true.
The  subsidized cost function $(c-\tilde f)$  is of the form
\begin{align}
  (c-\tilde f)(x,y) = \phi(x)+\psi (y)
\end{align}
and hence satisfies  the (strongest form of) constraint (\ref{S3}).

We have observed above that (\ref{S3}) implies that  $I_c =\int_{X\times Y} (c-\tilde f) \,d\tilde\pi$ for any finite transport plan $\tilde \pi\in \Pi(\mu,\nu)$. Thus $\int_{X\times Y} (c-\tilde f) \,d\tilde\pi=\int_{X\times Y} \phi(x)+\psi(y) \,d\tilde\pi(x,y)=J(\phi,\psi)$ yields that  $\phi, \psi$ are dual optimizers.

Finally, note that as a consequence of the fact that the total costs $\int_{X\times Y} \tilde f\, d\pi$ of our subsidy  system cannot be less than $\alpha$ together with the point-wise inequality $f\geq \tilde f$ implies that $f$ and $\tilde f$ coincide on the support of $\pi$.


\subsection{Existence of subsidy functions}

After the previous heuristic arguments we now pass to a more
rigorous analysis. In order to find dual maximizers, we will first
prove that there exists a subsidy function $f$ which satisfies
(\ref{S3}) for a given finite transport plan $\pi$. Note that in the
subsequent Proposition we do not assume
 that $c$ is $\mu\otimes\nu$-a.e.\ finitely valued.
\begin{proposition}\label{MonotonicityViaOptimality} Let $X,Y$ be Polish spaces equipped with Borel probability measures $\mu,\nu$.
Let  $c:X\times Y\to [0,\infty]$ be Borel measurable, assume that
$\pi$ is a finite transport plan
 and set $\alpha=I_c[\pi]-I_c\geq 0$. Then there exists a function
  $f:X\times Y\to [0,\infty]$   such that $\int f\, d\pi =\alpha$
  and,
  for  all $(x_1, y_1),\ldots,(x_n,y_n)\in X\times Y$, \begin{equation}\sum_{i=1}^n
  c(x_{i+1},y_i)+f(x_i,y_i)-c(x_{i},y_i)\geq 0.\footnote{We prefer to write
  $+ f(x_i,y_i)-c(x_i,y_i)$ rather then $-(c(x_i,y_i)-f(x_i,y_i))$ in view
  of our convention $\infty-\infty=\infty$.}
  \end{equation}
\end{proposition}

The main ingredient in the proof Proposition
\ref{MonotonicityViaOptimality} is the following duality theorem due
to Kellerer  (see \cite[Lemma 1.8(a), Corollary 2.18]{Kell84} and
\cite[p 212]{Kell85}).

\begin{theorem}[Kellerer]\label{KellererDuality}
  Let $Z_1,\ldots,Z_n,n\geq 2$ be Polish spaces equipped with Borel probability
   measures $\pi_1,\dotsc,\pi_n$ and assume that
   $C:Z= Z_1\times\ldots\times Z_n\to[-\infty,\infty)$ is Borel measurable and
   that $b = \sup_Z C$ is finite.
  Set \begin{align}
I_C:=&\inf\left\{ \int_Z C\ d\kappa: \kappa\in\Pi(\pi_1,\dotsc,\pi_n)\right\},\\
S_C:=&\sup\left\{\sum_{i=1}^n\int_{Z_i} \phi_i \, d\pi_i:
C(z_1,\dotsc,z_n)\geq \sum_{i=1}^n \phi_i(z_i)  \right\},
\end{align} where $\phi_1,\ldots, \phi_n$ are Borel functions taking
values in $[-\infty, \infty]$.

 Then $I_C=S_C$.
\end{theorem}

We will use it in the following form:
\begin{corollary}\label{KellererCorollary} Let $Z$
be a Polish space equipped with a Borel probability measure $\pi$.
Let $e:Z^n\to [0,\infty]$ be a Borel measurable function such   that
$e(z_1,z_2,\ldots,z_{n-1},z_n)=e(z_2,z_3,\ldots,z_n,z_1)$ for all
$z_1,\ldots,z_n\in Z$ and let $\alpha\geq 0$. Assume that
\begin{align}\sup\left\{\int_{Z^n} e\, d\kappa:\kappa\in
\Pi(\pi,\ldots,\pi)\right\}\leq n\alpha.\end{align} Then, for
$\delta>0$, there exists a function $f:Z\to [0,\infty]$ such that
$e(z_1,\ldots,z_n)\leq f(z_1)+\ldots+f(z_n)$ and $\int
f\,d\pi<\alpha+\delta$.
\end{corollary}

\begin{proof}
Applying Kellerer's Theorem to the function $C=-e$, we find
functions $f_1,\ldots, f_n:Z\to [0,\infty]$ such that
$e(z_1,\ldots,z_n)\leq f_1(z_1)+\ldots+f_n(z_n)$ and $\int f_1\,
d\pi+\ldots+\int f_n\, d\pi <n(\alpha+\delta)$. Set
$f(z)=(f_1(z)+\ldots+f_n(z))/n$.  Then
\begin{align}
e(z_1,\ldots,z_n) &=\frac1n\sum_{k=0}^{n-1} e(\sigma^k(z_1,\ldots,z_n))\\
&\leq \frac1n\sum_{k=0}^{n-1} f_1(z_{1+k})+\ldots +f_n(z_{n+k})\\
&=\frac1n \sum_{k=0}^{n-1}
f_1(z_{1+k})+\ldots+f_n(z_{1+k})=\sum_{k=0}^{n-1} f(z_{1+k}),
\end{align}
where $\sigma(r_1,\ldots,r_n)=(r_2,\ldots,r_{n+1})$.
\end{proof}

\begin{lemma}\label{KellererMeasure}  Let $X,Y$ be Polish spaces
equipped with Borel probability measures $\mu,\nu$, let $c:X\times
Y\to [0,\infty]$ be Borel measurable and assume that $\pi$ is a
finite transport plan satisfying $ I_c[\pi]\leq I_c+\alpha$. Set
\begin{align} e(x_1,y_1,\ldots x_n,y_n):=\Big(\sum_{i=1}^n
c(x_{i+1},y_i)-c(x_i,y_i)\Big)_{-}\geq 0.\end{align} Then
\begin{align}\label{BoundToImprovement}
\sup\left\{\int_{(X\times Y)^n} e\, d\kappa:\kappa\in
\Pi(\pi,\ldots,\pi)\right\}\leq n\alpha.
\end{align}
\end{lemma}

\begin{proof}
Denote by $\sigma, \tau: (X\times Y)^n  \to (X\times Y)^n$ the mappings 
\begin{align}
\sigma: (x_i,y_i)_{i=1}^n  &\mapsto  (x_{i+1},y_{i+1})_{i=1}^n \\
  \tau: (x_i,y_i)_{i=1}^n  &\mapsto \hspace{2ex}(x_{i},y_{i+1})_{i=1}^n.
\end{align}
Observe that $\sigma^n=\tau^n=\mbox{Id}_{(X\times Y)^n}$ and that $\sigma$ and $\tau$ commute.  By $p_i: (X\times Y)^n \to X\times Y, (x_1,y_1,\ldots, x_n,y_n)\mapsto (x_i,y_i)$ we denote the projection on the $i$-th component of the product, while the projections $p_X: X\times Y \to X, (x,y)\mapsto x$ and $p_Y: X\times Y \to Y, (x,y)\mapsto y$ are defined as above.

Pick $\kappa \in \Pi(\pi,\ldots,\pi)$. By replacing $\kappa$ with 
\begin{align}
\frac1n
\big(\kappa +\sigma\push \kappa+\ldots+(\sigma)^{n-1}\push\kappa\big),
\end{align}
   we may assume that $\kappa$ is $\sigma$-invariant.  Set
$B=\{e>0\}$ and consider the restriction of $\kappa$ to $B$ defined
by $ \tilde \kappa(A)= \kappa(A\cap B)$ for Borel sets $A\subseteq
(X\times Y)^n$. $\tilde \kappa$ is $\sigma$-invariant since both the
measure $\kappa$ and the set $B$ are $\sigma$-invariant. Denote the
marginal of $\tilde \kappa$ in the first coordinate $(X\times Y)$ of
$(X\times Y)^n$ by $\tilde \pi$. Due to $\sigma$-invariance we have
\begin{align} p_i\push\tilde \kappa =p_i\push(\sigma\push\tilde \kappa)\\
=(p_i\circ \sigma)\push\tilde \kappa=p_{i+1}\push\tilde \kappa,\end{align} i.e.\
all marginals coincide and we have $\tilde \kappa \in
\Pi(\tilde \pi,\ldots,\tilde \pi)$. Furthermore, since
$\tilde \kappa\le\kappa$, the same is true for the marginals, i.e.\
$\tilde \pi\le \pi$. Denote the marginal $p_1\push (\tau\push\tilde  \kappa)$ of $\tau\push\tilde  \kappa$ in
the first coordinate $(X\times Y)$ of $(X\times Y)^n$ by
$\tilde \pi_{\beta}$. As $\sigma$ and $\tau$ commute,
$\tau\push\tilde \kappa$ is  $\sigma$-invariant, so the marginals in
the other coordinates coincide with $\tilde \pi_{\beta}$. 
Note that $\tilde  \pi_\beta(X\times Y)= \tilde \kappa\big((X\times Y)^n\big)=\tilde  \pi (X\times Y).$ Moreover  $\tilde {\pi}$ and $\tilde \pi_{\beta}$ have  the
same marginals in $X$ resp.\ $Y$. Indeed, let $C\subseteq X,D\subseteq Y$ be
Borel sets. Then
\begin{align}
p_X\push\tilde \pi_{\beta}(C)&=
\tau\push\tilde \kappa ((p_X\circ p_1\circ\tau)^{-1}[C])=\\
&=\tau\push\tilde \kappa\{(x_1,\ldots,y_n):x_1\in C\}=\\
&= \tilde  \kappa \{(x_1,\ldots,y_n):x_1\in C\}=p_X\push\tilde \pi(C)\\
p_Y\push\tilde \pi_{\beta}(D)&=
\tau\push\tilde \kappa ((p_Y\circ p_1\circ\tau)^{-1}[D])=\\
&=\tau\push\tilde \kappa\{(x_1,\ldots,y_n):y_1\in D\}=\\
&= \tilde  \kappa \{(x_1,\ldots,y_n):y_2\in D\}=p_Y\push\tilde \pi(D).
\end{align}
This enables us to define an improved transport plan by
\begin{equation}\label{BetterMeasure}
\pi_{\beta}=(\pi-\tilde \pi)+\tilde \pi_{\beta}.
\end{equation}
Since $\tilde \pi\leq \pi$, we have that $(\pi-\tilde \pi)$ is a positive measure, hence (\ref{BetterMeasure}) defines a positive measure as well. Since $\tilde \pi$ and $\tilde \pi_{\beta}$ have the same total mass, $\pi_{\beta}$ is a probability measure. Furthermore $\tilde \pi$ and $\tilde \pi_{\beta}$ have the same marginals in $X$, resp.\ $Y$, so $\pi_{\beta}$ is indeed a transport plan.
It remains to apply the assumption that the transport costs of $\pi_\beta$ cannot be cheaper by more than $\alpha$  than the ones of $\pi$:
\begin{align}
\alpha &\geq I_c[\pi]-I_c[{\pi_\beta}]\\
&=\int_{X\times Y} c\, d(\tilde \pi-\tilde \pi_\beta)\\
&=\frac1n \sum_{i=1}^n \int_{(X\times Y)^n} c \circ p_i\, d (\tilde \kappa-\tau\push\tilde \kappa)\\
&=\frac1n \sum_{i=1}^n \int_{(X\times Y)^n} c(x_i,y_i)-c(x_{i+1},y_i)\, d \tilde \kappa(x_1,\ldots y_n)\\
&=\frac1n \sum_{i=1}^n \int_{B} c(x_i,y_i)-c(x_{i+1},y_i)\, d \kappa(x_1,\ldots y_n)\\
&= \frac1n\int_B e(x_1,\ldots,y_n)\, d\kappa\\
&= \frac1n\int_{(X\times Y)^n} e(x_1,\ldots,y_n)\, d\kappa.
\end{align} Since $\kappa\in\Pi(\pi,\ldots,\pi)$ was arbitrary, this yields (\ref{BoundToImprovement}).
\end{proof}

In the proof of Proposition \ref{MonotonicityViaOptimality} we shall apply a result of Koml\'os (\cite{Koml67}).
\begin{lemma}
\label{komlos} Let $(f_{n})_{n\geq1}$ be a sequence of measurable
$[0,\infty]$-valued functions on a probability space $(Z,\pi)$ such that $\sup_{n\geq 1} \| f_n\|_1<\infty$.
 Then there exists a subsequence $(\tilde f_n)_{n\geq 1}$ such that the functions \begin{align}
 \frac1n\big(\tilde f_1+\ldots+\tilde f_n\big), \quad
n\geq 1\end{align} converge $\pi$-a.e.\ to a function taking values in $[0,\infty]$. 

In particular, there exist functions  $g_n\in\operatorname*{conv}(f_{n},f_{n+1},\dots)$ such that $(g_n)_{n\geq 1}$ converges $\pi$-a.e.. 
\end{lemma}
\begin{proof}
The first part of Lemma \ref{komlos} is Koml\'os' original result which we will not prove.
The assertion that there exist 
$g_n\in\operatorname*{conv}(f_{n},f_{n+1},\dots)$ such that $(g_n)_{n\geq 1}$ converges almost everywhere\footnote{In fact, this result holds true without \emph{any} integrability assumptions, see \cite[Lemma A1.1]{DeSc94}.} is a simple consequence which we will derive for the sake of completeness. Assume that  $\frac1n\big(\tilde f_1+\ldots+ \tilde f_n\big)$ converges $\pi$-a.e.\ to a function $g$. Since all functions $f_n, n\geq 1$ are $\pi$-a.e.\ finitely valued, for each $k\geq 1$ there exists some $n_k$ such that 
\begin{align}
\pi\left(\left\{\frac 1n_k \big(\tilde f_1+\ldots+\tilde f_k\big)>\frac 1k\right\}\right)<\frac1{2^k}. 
\end{align} 
Set $g_k=\frac{1}{n_k}\big(\tilde f_{k+1}+\ldots+\tilde f_{n_k+k}\big)$.
\end{proof}

\begin{proof}[Proof of Proposition \ref{MonotonicityViaOptimality}.]
Combining Lemma \ref{KellererMeasure} and
Corollary \ref{KellererCorollary} we achieve that for
each $n\geq 2$ there exists a function $f_n$ such that
$\int f_n\, d\pi \leq\alpha+1/n$ and  for  all
 $(x_1, y_1),\ldots,(x_n,y_n)\in X\times Y$
\begin{align}\label{CmonN}
\sum_{i=1}^n c(x_{i+1},y_i)+f_n(x_i,y_i)-c(x_{i},y_i)\geq 0.\end{align}
Observe that for all $p\geq 1$ the function $f_{np}$ satisfies
(\ref{CmonN}) as well since we can run through the cycle
\begin{align} (x_1,y_1)\to \ldots \to (x_n,y_n)\to (x_1,y_1)
\to\ldots\end{align} $p$ times. Also note that any convex
combination and any pointwise limit of functions which satisfy
(\ref{CmonN}) for some fixed $n$ satisfies (\ref{CmonN}) (for the
same $n$) as well. Thus we may apply Lemma \ref{komlos} to find
functions $g_n\in\operatorname*{conv}(f_{n!},f_{(n+1)!},\dots)$
which converge $\pi$-almost everywhere. Defining $g$ as the pointwise limit where this limit exists and $\infty$ elsewhere, yields a function which satisfies
(\ref{CmonN}) for every $n$. Moreover
\begin{align}\int g\,d\pi=\int \liminf_{n\to \infty} g_n\, d\pi\leq
\liminf_{n\to \infty} \int g_n\, d\pi=\alpha,\end{align}
hence the desired bound for the total costs of subsidy holds true.
\end{proof} 

\subsection {Dual maximizers by subsidized transport plans}
The main goal of this section is to prove the sandwich-type result announced above.
\begin{proposition}\label{SandwichTheorem}
Assume that $X,Y$ are Polish spaces equipped with Borel probability measures $\mu,\nu$, that $\ovlc:X\times Y\to (-\infty,\infty]$ is Borel measurable and $\mu\otimes \nu$-a.e.\ finite and that $\unlc:X\times Y\to [-\infty,\infty)$ is Borel measurable. If
\begin{equation}\label{SandwichMonotonicity}
\sum_{i=1}^n \ovlc(x_{i+1},y_i)-\unlc (x_i,y_i)\geq 0
\end{equation} for all $x_1,\ldots, x_n\in X$, $y_1,\ldots,y_n\in Y$, there exist Borel measurable functions $\phi: X\to [-\infty,\infty),\psi: Y\to [-\infty,\infty)$ and Borel sets $X'\subseteq X, Y'\subseteq Y$ of full measure  such that
\begin{align}
\unlc(x,y)\leq \phi(x)+\psi (y)\leq \ovlc (x,y),
\end{align}
where the lower bound holds for $x\in X',y\in Y'$ and the upper bounded is valid  for all $x\in X, y\in Y$.
\end{proposition}

Observe that Proposition  \ref{SandwichTheorem} is in fact a generalization of Proposition \ref{CmonToStrongCmon}: Let $\pi\in \Pi(\mu,\nu)$ be a finite transport plan which is concentrated  on a $c$-cyclically monotone Borel set $\Gamma$. Without loss of generality assume that $c$ is finite on $\Gamma$. Set $\ovlc=c$ and
\begin{align}
\unlc(x,y):=\left\{
\begin{array}{cl}
c(x,y)&\mbox{ if }(x,y)\in \Gamma\\
-\infty&\mbox{ else }
\end{array}
\right..
\end{align} Then $\ovlc,\unlc$
satisfy the assumptions of Proposition \ref{SandwichTheorem} and thus there exist Borel measurable functions $\phi,\psi$ which satisfy $\phi(x)+\psi(y)\leq c(x,y)$ for all $x\in X,y\in Y$ and   $\phi(x)+\psi(y)\geq c(x,y)$ for $\pi$-almost all $(x,y)\in X\times Y$, hence $\pi$ is strongly $c$-cyclically monotone.

\medskip


Before giving the proof of Proposition \ref{SandwichTheorem} we need
some preliminaries, in particular we will recapitulate some facts
from the theory of analytic sets. They will be needed  to deal with
certain measurability issues which arise in the course of the proof.

 Let $X$ be a
Polish space. A set $A\subseteq X$ is \emph{analytic} if there exist
a Polish space $Z$, a Borel measurable function $f:Z\to X$ and a
Borel set $B\subseteq Y$ such that $f(B)=A$.

\begin{lemma}\label{InfIsMeasurable}
Let $X,Z$ be Polish spaces and  $g:X\times Z\to [-\infty, \infty]$ a
Borel measurable function. Set
\begin{align}
\phi(x)=\inf_{z\in Z} g(x,z).
\end{align}
Then $\{\phi<\alpha\}$ is analytic for every $\alpha\in
[\infty,\infty]$.
\end{lemma}
\begin{proof}
We have
\begin{align}
\phi(x)<\alpha\ \Longleftrightarrow\ \exists z, g(x,z)< \alpha.
\end{align} Thus $\{\phi<\alpha\}=p_X[\{g(x,z)<\alpha\}]$.
\end{proof}

 Given a Borel measure $\mu$ on $X$, we
denote its completion by $\tilde \mu$. By a result of Luzin (see for
instance \cite[Theorem 21.10]{Kech95}) every analytic set $A\subset
X$ is the disjoint union of a Borel set and a $\tilde \mu$-null-set. This will allow us to replace a function which only
satisfies that the sets $\{\phi<\alpha\}, \alpha\in [-\infty,
\infty]$ are analytic by a Borel measurable function.

\begin{lemma}\label{borelversion} Let $X$ be a Polish space and $\mu$ a finite Borel measure on $X$.
If $\phi: X\to [-\infty, \infty]$ satisfies that
$\{\phi(x)<\alpha\}$ is analytic for each $\alpha\in
[-\infty,\infty]$, then there exists a Borel measurable function
$\tilde{\phi}: X\to[-\infty,\infty]$ such that $\tilde{\phi}
\leq\phi$ everywhere and $\phi=\tilde{\phi}$ almost everywhere with
respect to $\tilde \mu$.
\end{lemma}
\begin{proof}
Let $(I_n)_{n=1}^{\infty}$ be an enumeration of the intervals
$[-\infty,\alpha),\alpha\in\Q$. Then for each $n\in \N$,
$\phi^{-1}[I_n]$ is $\tilde{\mu}$-measurable and hence the union of
a Borel set $B_n$ and a $\tilde{\mu}$-null set $N_n$. Let $N$ be a
Borel null set which covers $\bigcup_{n=1}^{\infty} N_n$. Let
$\tilde{\phi}(x)=\phi(x)-\infty\cdot \mathbf{1}_N(x)$. Clearly
$\tilde{\phi}(x) \le \phi(x)$ for all $x\in X$ and
$\phi(x)=\tilde{\phi}(x)$ for $\tilde\mu$-almost all $x\in X$.
Furthermore, $\tilde{\phi}$ is Borel measurable since
$(I_n)_{n=1}^{\infty}$ is a generator of the Borel $\sigma$-algebra
on $[-\infty,\infty)$ and for each $n\in \N$ we have that
$\tilde{\phi}\me [I_n]=B_n\setminus N$ is a Borel set.
\end{proof}

\begin{lemma}\label{Connectedness}
Assume that in the setting of Proposition \ref{SandwichTheorem} we have that
\begin{align}
\tilde\mu(p_X[\{\unlc >-\infty\}\setminus (X\times N)])=1
\end{align}
 for every $\nu$-null-set $N\subseteq Y$. Then there exist $x_1\in X$ and a Borel set $X'\subseteq X$ with $\mu(X')=1$ such that for each $x\in X'$ there are $y_1,y\in Y$ satisfying
\begin{align}
\unlc(x_1,y_1)>-\infty, & \quad \ovlc(x,y_1)<\infty,\\
\unlc(x,y)>-\infty, & \quad \ovlc(x_1,y)<\infty.
\end{align}
\end{lemma}
\begin{proof}
Set
\begin{align}
X_1 & =\{x:\nu(\{y:\ovlc(x,y)<\infty\})=1\}\\
Y_1 & =\{y:\mu(\{x:\ovlc(x,y)<\infty\})=1\}.
\end{align}
By Fubini's Theorem, $\mu(X_1)=\nu(Y_1)=1$. Since $\tilde\mu(p_X[\{\unlc >-\infty\}\cap (X\times Y_1)])=1$, there exist $x_1\in X_1, y_1\in Y_1$ such that $\unlc(x_1,y_1)>-\infty$. Since $x_1\in X_1$, 
\begin{align}
Y':=\{y\in Y_1:\ovlc(x_1,y)<\infty\}
\end{align} has $\nu$-measure one. Consequently
\begin{align} X' :=\{x:\ovlc(x,y_1)<\infty\}\cap p_X [\{\unlc >-\infty\}\cap (X\times Y')]\end{align}
has full  $\tilde\mu$-measure. It remains to check that the assertions of Lemma \ref{Connectedness} are satisfied. Choose $x\in X'$. By definition of $X'$, $\ovlc(x,y_1)<\infty$. Since $x\in p_X [\{\unlc >-\infty\}\cap (X\times Y')] $, there exists some $y\in Y'$ such that $\unlc (x,y)>-\infty$. Since $y\in Y'$, $\ovlc(x_1,y)<\infty$.
\end{proof}

\begin{proof}[Proof of Proposition \ref{SandwichTheorem}]
Note that it is sufficient to define $\phi$ and $\psi $ on Borel
sets $X'\subseteq X,Y'\subseteq Y$ with $\mu(X')=\nu(Y')=1$, since
they can then be extended to $X$ and $Y$ by setting them $-\infty$
on the null-sets $X\setminus X', Y\setminus Y'$. This will be used
several times in the course of the proof.

Next we will show  that  it is sufficient to consider the case that
$\tilde\mu(p_X[\{\unlc >-\infty\}\setminus (X\times N)])=1$ for
every $\nu$-null-set $N\subseteq Y$, such that Lemma
\ref{Connectedness} is applicable. Set
\begin{align}
\beta:=\inf \{\tilde\mu(p_X[\{\unlc >-\infty\}\setminus (X\times
N)]):N\subseteq Y, \nu(N)=0 \}.
\end{align}
Choose for each $k\geq 1$ a $\nu$-null-set $N_k\subseteq Y$ such that
\begin{align}
\beta +\frac1k\geq\tilde\mu(p_X[\{\unlc >-\infty\}\setminus (X\times N_k)])
\end{align}
Set  $N:=\bigcup_{k=1}^\infty N_k, Y'=Y\setminus N$ and let $M\subseteq p_X[\{\unlc >-\infty\}\setminus (X\times N)]$ be a Borel set with $\mu(M)=\beta$. Then
$\tilde\mu(M\cap p_X[\{\unlc >-\infty\}\setminus (X\times N')])=\beta $ for every $\nu$-null-set $N'\subseteq Y'$ and it is sufficient to define $\phi,\psi$ on $M$ and $Y'$; as above they can then
be extended to $X$ and $Y$ by setting them $-\infty$ on $X\setminus M$ resp.\ $N$. Thus we may assume without loss of generality  that $M=X$ or, equivalently,  that $p_X[\{\unlc >-\infty\}\setminus (X\times N)]$ has full measure for every $\nu$-null-set $N\subseteq Y'$.

Choose $x_1$ and $X'$ according to Lemma \ref{Connectedness} and set

\begin{align}
\phi_n(x)=\inf \left\{\sum_{i=1}^n \ovlc(x_{i+1},y_i)-\unlc(x_i,y_i):x_{n+1}=x\right\}
\end{align}
\begin{align}\mbox{ and } \phi(x)=\inf_{n\geq1} \phi_n(x).
\end{align}

 Fix  $x\in X'$. To see that $-\infty<\phi(x)<\infty$, pick $y,y_1\in Y'$ according to Lemma \ref{Connectedness} such that $\unlc(x_1,y_1)>-\infty,  \ovlc(x,y_1)<\infty,
\unlc(x,y)>-\infty, \ovlc(x_1,y)<\infty$.  Then 
\begin{align}\phi(x)\leq\phi_1(x)\leq \ovlc(x,y_1)-\unlc(x_1,y_1)<\infty.\end{align} To prove the lower bound, we pick $n\geq 1$ and $x_2,\ldots, x_n\in X', y_1,\ldots, y_n\in Y'$ and set $x_{n+1}=x, y_{n+1}=y, x_{n+2}=x_1$. By (\ref{SandwichMonotonicity}),
\begin{align}
\sum_{i=1}^{n+1} \ovlc(x_{i+1},y_i)-\unlc (x_i,y_i) & \geq 0\\
\mbox{whence}\quad \sum_{i=1}^{n} \ovlc(x_{i+1},y_i)-\unlc (x_i,y_i) & \geq \underbrace{-\big[\ovlc(x_{n+2},y_{n+1})-\unlc (x_{n+1},y_{n+1})\big]}_{=-\big[\ovlc(x_1,y)- \unlc(x,y)\big]}.
\end{align}
Taking the infimum over all possible choices of  $n\geq 1$ and $x_2,\ldots, x_n\in X', y_1,\ldots, y_n\in Y'$, we achieve that $\phi(x) \geq-\big[\ovlc(x_1,y)- \unlc(x,y)\big]>-\infty.$

\medskip

Next observe that for $x,x'\in X', y\in Y'$ and $n\geq 1$
\begin{align}
\phi_{n+1}(x)\leq &\, \inf \left\{\sum_{i=1}^{n+1} \ovlc(x_{i+1},y_i)-\unlc(x_i,y_i):x_{n+2}=x,x_{n+1}=x', y_{n+1}=y\right\}\\
= &\, \phi_n(x')+ \big[\ovlc(x,y)-\unlc(x',y)\big]\label{RelevantInequality}
\end{align}
Taking the infimum over $n\geq 1$ yields that
\begin{align}
\phi(x)\leq \inf_{n\geq 1} \phi_{n+1}(x) & \leq \phi(x')+ \big[\ovlc(x,y)-\unlc(x',y)\big]\\
 \Longrightarrow \qquad \unlc(x',y)-\phi(x') & \leq \ovlc (x,y)-\phi(x).
\end{align}

At this point we will take care about measurability of $\phi$. First apply
Lemma \ref{InfIsMeasurable} to the spaces $X'$ and $Z=\bigcup_{n=1}^\infty (Y'\times X)^n$ to see that $\{\phi<\alpha\}$ is analytic for each $\alpha\in [-\infty,\infty]$.
Then  we may shrink
 $X'$ a little bit to achieve, by  Lemma \ref{borelversion}, that $\phi$ is
even Borel measurable.  Note that (\ref{RelevantInequality}) is then still valid for all $x,x'\in X',y\in Y'$.
\medskip

 By Fubini's Theorem, for $\nu$-almost all $y\in Y'$ there exists some $x\in X$ such that $\ovlc(x,y)<\infty$. By shrinking $Y'$ a little bit if necessary, we may assume that this is the case for all $y\in Y'$. Then the function
\begin{align}\label{SecondRelevant}
\psi(y):=\inf_{x\in X'} \ovlc(x,y)-\phi(x)
\end{align}
is finitely valued on $Y'$. As above, we apply Lemma
 \ref{InfIsMeasurable} and Lemma \ref{borelversion} and shrink
 $Y'$ a little further to achieve that $\psi$ is Borel measurable.
Moreover
\begin{align}
\unlc(x,y)\leq \phi(x)+\psi (y)\leq \ovlc (x,y)
\end{align}
holds for all $x\in X',y\in Y'$ by (\ref{RelevantInequality}) and (\ref{SecondRelevant}).
\end{proof}

\begin{proof}[Proof of Theorem \ref{DualMaximizers}] Pick an arbitrary finite transport plan $\pi\in\Pi(\mu,\nu)$. Choose  a subsidy function $f:X\times Y\to[0,\infty]$ according to Proposition \ref{MonotonicityViaOptimality} such $\int f\,d\pi\leq I_c[\pi]-I_c$. Set $\ovlc(x,y)=c(x,y)$ and
\begin{align}
\unlc(x,y):=\left\{
\begin{array}{cl}
c(x,y)-f(x,y) &\mbox{if $c(x,y)\neq \infty$}\\
-\infty &\mbox{if $c(x,y)= \infty$}
\end{array}\right.,
\end{align}
in particular  $\unlc=c-f$ holds $\pi$-a.e.\ since $c$ is
$\pi$-a.e.\ finite. For all $x_1,\ldots, x_n\in X$ and $y_1,\ldots,
y_n\in Y,$
\begin{align}
&\sum_{i=1}^n \ovlc(x_{i+1},y_i)-\unlc(x_i,y_i)\geq\\
&\sum_{i=1}^n c(x_{i+1},y_i)+f(x_i,y_i)-c(x_i,y_i)\geq 0,
\end{align}
 Thus by Proposition \ref{SandwichTheorem}, there exist functions $\phi:X\to [-\infty,\infty)$, $\psi:Y\to [-\infty, \infty)$ such that $\phi(x)+\psi(y)\leq \ovlc(x,y)=c(x,y)$ for all $x\in X,y\in Y$ and $\phi(x) + \psi(y)\geq \unlc (x,y)=c(x,y)-f(x,y)$ for $\pi$-almost all $(x,y)\in X\times Y$. This implies
\begin{align}
I_c\geq J(\phi,\psi)\geq \int c(x,y)-f(x,y)\, d\pi(x,y)\geq  I_c[\pi]-\big(I_c[\pi]-I_c\big)= I_c,
\end{align} thus $I_c=J(\phi,\psi)$ and hence  $\phi,\psi$ are dual maximizers.
\end{proof}

To conclude the proof of Theorem \ref{BorelDuality} we show Lemma \ref{IndependenceOfPi}.

\begin{proof}[Proof of Lemma \ref{IndependenceOfPi}.]
Similarly as in \cite{ScTe08} we define
\begin{align*}
\phi_{n}(x)  &  =(-n\lor\phi(x))\land n,\\
\psi_{n}(y)  &  =(-n\lor\psi(y))\land n,\\
\xi_{n}(x,y)  &  =\phi_{n}(x)+\psi_{n}(y),\\
\xi(x,y)  &  =\phi(x)+\psi(y),
\end{align*}
for $(x,y)\in X\times Y$ and $n\geq1$. Observe that
$\xi_{n}\uparrow\xi$ on $\{\xi\geq0\}$ and $\xi_{n}\downarrow\xi$ on
$\{\xi\leq0\}$, as $n\to\infty$. Moreover $\xi_n\leq c$ so that
$(\phi_n,\psi_n)\in\Phi(\mu,\nu)$.

Additionally, $ \int\xi\,d\pi \leq \int c\,d\pi, \int \xi \,d
\tilde\pi<\int c\,d\tilde\pi$ exist, taking possibly the value
$-\infty$, since $\xi\leq c$ holds $\pi$-almost surely  as well as
$\tilde\pi$-almost surely and we assume that $\int c\,d\pi,\int c\,
d\tilde \pi<\infty$. By the assumption on equal marginals of $\pi$
and $\tilde\pi$ we obtain
\begin{align}
\int \xi_n\,d\pi  &  =\int \phi_{n}\, d\pi+\int \psi_{n}\,d\pi\\
&  =\int \phi_{n}\, d\tilde\pi+\int \psi_{n}\,d\tilde\pi\\
&  =\int\xi_{n}\,d\tilde\pi,
\end{align}
for $n\geq0$, hence
\begin{align}
\int\xi_{n}1_{\{\xi\geq0\}}+\xi_{n}1_{\{\xi\leq0\}}\,d\pi = \int \xi
_{n}1_{\{\xi\geq0\}}+\xi_{n}1_{\{\xi\leq0\}}\,d\tilde\pi.
\end{align}
By our previous considerations we can pass to the limits and obtain
$\int\xi\,d\pi=\int\xi\,d\tilde\pi$. Indeed the limits are monotone
on $\{\xi\geq0\}$ and on $\{\xi\leq0\}$ and the convergence is
dominated by $c$ on $\{\xi\geq 0\}$. Hence the limits of $\int
\xi_{n}\,d\pi=\int\xi_{n}\,d\tilde\pi$ exist as $n\to\infty$ and are
equal. Consequently $\int\xi\,d\pi=\int \xi\,d\tilde\pi$.
\end{proof}

\section{Examples}
We start with a simple example which shows that Monge-Kantorovich
duality does not hold in general for a measurable cost function
$c:X\times Y\to[0,\infty]$.

\begin{example}\label{DualityGapExample}\label{ZeroOneInfty}
Let $X=Y=[0,1]$, $\mu=\nu$ the Lebesgue-measure and set
\begin{align}
c(x,y)=\left\{
\begin{array}{cl}
\infty &\mbox{ for }x<y\\
1&\mbox{ for }x=y\\
0&\mbox{ for }x>y
\end{array}\right.
\end{align} for $(x,y)\in X\times Y$.
The optimal (and in fact the only finite) transport plan $\pi$ is
concentrated on the diagonal and yields costs of one. Assume that
$\phi,\psi:[0,1]\to [-\infty,\infty)$ are integrable functions
satisfying $\phi(x)+\psi(y)\leq c(x,y)$ for all $x,y\in [0,1]$. Then
\begin{align}
\int \phi\,d\mu+\int \psi\,d\nu &= \lim_{\alpha\downarrow
0}\int_0^{1-\alpha}\phi(x+\alpha)+\psi(x)\,dx\\
&\leq\lim_{\alpha\downarrow 0}\int_0^{1-\alpha}c(x+\alpha,x)\,dx=0.
\end{align}
Thus there is a duality gap, i.e.,
\begin{align}
1=I_c>\sup\{J(\phi,\psi):(\phi,\psi)\in \Phi(\mu,\nu)\}=0.
\end{align}
Note also that $I_c$ fails to be continuous in the sense of Corollary \ref{IcContinuous}: For each $n\geq 1$, there exist transport plans assigning arbitrarily small costs to the function  $c\wedge n$ such that $\lim_{n\to \infty}I_{c\wedge n}= 0\neq 1 =I_c$.
\end{example}

In \cite[Example 5.3]{BGMS08} a certain  variant of Example
\ref{DualityGapExample} is considered. By setting $c(x,y)=\sqrt{x-y}
$ for $x>y$, the cost function becomes lower semi-continuous. In
this case duality does hold true, but there are no dual maximizers,
that is, the optimal transport plan $\pi$ is not strongly
$c$-cyclically monotone. We present here yet another variant of
 Example \ref{DualityGapExample} which shows that dual
maximizers need not exist, even if the cost function is assumed to
be continuous.

\begin{example}\label{ContinuousNoMaximizers}
Let $X=Y=\N\cup\{\omega\}$ where we take  $\omega$ to be a
``number'' larger than all $n\in \N$.  Equip $X$ and $Y$ with the discrete
topology and define $\mu= \nu$ such that positive measure is assigned
to each point in $X$, resp.\ $Y$. Set
\begin{align}
c(x,y)=\left\{
\begin{array}{cl}
\infty &\mbox{ for }x<y\\
1&\mbox{ for }x=y\\
0&\mbox{ for }x>y
\end{array}\right.
\end{align}  for $(x,y)\in X\times Y$.  As above we find that the only finite transport plan $\ppi$ is concentrated on the diagonal and yields costs of $1$. Since $X$ and $Y$ are discrete spaces, $c$ is continuous  with respect to the product topology on $X\times Y$ and hence duality holds true.  Striving for a contradiction, we assume that there exist dual maximizers $\pphi,\ppsi:\N\cup\{\omega\}\to [-\infty,\infty)$. Note that $\pphi$ and $\ppsi$ are necessarily finitely valued  since $X$ and $Y$ have no non trivial null-sets. Since $\pphi,\ppsi$ witness that $\ppi$ is strongly $c$-monotone we have
\begin{align}
\pphi(n)+\ppsi(n)=1, \pphi(n+1)+\ppsi(n)\leq 0, \mbox{ and } \pphi(\omega)+\ppsi(n)\leq 0
\end{align} for each $n\in \N$. This yields
\begin{align}
\pphi(n+1)\leq \pphi(n)-1 , \mbox{ and } \pphi(\omega)\leq \pphi(n)-1
\end{align} for all $n\in \N$  which is impossible for a finitely valued function.
\end{example}

One can try to overcome the difficulties encountered in Examples
\ref{DualityGapExample}, \ref{ContinuousNoMaximizers} and
\cite[Example 5.3]{BGMS08} by admitting dual optimizers from a
larger class of mappings: Consider functions   $\phi:X\to
[-\infty,\infty),\psi:Y\to [-\infty,\infty)$  which do not
necessarily satisfy the inequality $\phi(x)+\psi (y) \leq c(x,y)$
for \emph{all} $(x,y)\in X\times Y$ but do satisfy the potentially
weaker condition
\begin{align}\label{WeakCondition}
\mbox{$\phi(x)+\psi(y)\leq c(x,y)$, $\pi$-a.e.\ for
   every \emph{finite} transport plan $\pi\in \Pi(\mu,\nu)$.}
\end{align}
It follows then that $J(\phi,\psi)=\int_{X\times Y} \phi(x)+\psi(y)\, d\pi(x,y)\leq I_c[\pi]$ for each finite transport plan $\pi\in \Pi(\mu,\nu)$, such that
\begin{align}
J(\phi,\psi)\leq \inf_{\pi\in\Pi(\mu,\nu)} I_c[\pi]=I_c.
\end{align}
 Therefore it is reasonable to consider all pairs of functions $(\phi,\psi)$ satisfying (\ref{WeakCondition}) as admissible solutions of the dual part of the Monge-Kantorovich problem.
In particular a transport plan $\pi$ is optimal provided that there exist
measurable (not necessarily integrable) functions $\phi:X\to
[-\infty,\infty),\psi:Y\to[-\infty,\infty)$ satisfying (\ref{WeakCondition}) and
\begin{align}\label{ItMaximizes}
\int_{X\times Y} \big[\phi(x)+\psi(y)\big] \,d
\pi(x,y)=\int_{X\times Y} c\,d\pi.
\end{align}
Clearly, (\ref{ItMaximizes}) is tantamount to requiring that
 $\phi(x)+\psi(y)=c(x,y)$ for $\pi$-almost all $(x,y)$.
Observe that dual optimizers in this weak sense exist  in Examples \ref{DualityGapExample} and \ref{ContinuousNoMaximizers}.
Since the only finite transport plan is the optimal one, we may just take $\phi\equiv \psi \equiv 1/2$. However the subsequent construction (which is a variation
of \cite[Example 3.5]{AmPr03}) shows that, in general, dual
optimizers do not even exist in this weak sense.
\begin{example}\label{ExAmPra}
Let $X=Y= [0,1]$, equipped with Lebesgue measure $\lambda=\mu=\nu$.
Pick $\alpha\in [0,1)$ irrational. Set $$\Gamma_0=\{(x,x):x\in
X\},\quad\Gamma_1=\{(x,x\oplus\alpha):x\in X\},$$ where $\oplus$ is
addition modulo $1$. Define $c: X\times Y \to \Rop$ by
\begin{align}
c(x,y)=\left\{
\begin{array}{cl}
0&\mbox{ for }(x,y)\in\Gamma_0, x\in [0,1/2]\\
2&\mbox{ for }(x,y)\in\Gamma_0, x\in (1/2,1)\\
1&\mbox{ for }(x,y)\in\Gamma_1\\
\infty&\mbox{ else }
\end{array}\right..
\end{align} Note that $c$ is lower semi-continuous whence duality holds (cf.\ \cite[Theorem 2.2]{Kell84}).
For $i=0,1$, let $\pi_i$ be the obvious transport plan concentrated
on $\Gamma_i$. Then all finite transport plans are given by  convex
combinations of the form $\rho\pi_0+(1-\rho)\pi_1, \rho\in[0,1]$ and
each of these transport plans leads to costs of $1$. Assume that
$\pphi,\ppsi:[0,1)\to [-\infty,\infty)$ are measurable functions
which satisfy
\begin{enumerate}
\item $J(\pphi,\ppsi)=1$,
\item $\pphi(x)+\ppsi(y)\leq c(x,y)$ for $\pi_0$- and $\pi_1$-almost all $(x,y)$.
\end{enumerate}
This implies that, in fact, $\pphi(x)+\ppsi(y)=c(x,y)$ for $\pi_0$-
and $\pi_1$-almost all $(x,y)$. Thus
\begin{align}
\pphi(x)+\ppsi(x)=\left\{
\begin{array}{cl}
0&\mbox{ for } x\in [0,1/2]\\
2&\mbox{ for } x\in (1/2,1)
\end{array}\right., \quad \pphi(x)+\ppsi(x\oplus\alpha)= 1
\end{align}
\begin{align}\label{LotsOfVariation}
\mbox{whence}\quad  \ppsi(x\oplus\alpha)=\left\{
\begin{array}{cl}
\ppsi(x)+1&\mbox{ for } x\in [0,1/2]\\
\ppsi(x)-1&\mbox{ for } x\in (1/2,1)
\end{array}\right.
\end{align}
for all $x\in I$, where $I\subseteq [0,1]$ is a Borel set of measure
$1$. By passing to a subset of full measure, we may additionally
assume that $I\oplus\alpha=I$.  Pick a set $B\subseteq I$ such that
$\lambda(B)>0$ and $\sup_{x\in B} \ppsi(x)-\inf_{x\in B} \ppsi(x)<1$.
By a classic theorem of Steinhaus, $B-B$ contains a non-empty open set.
Since $\alpha$ is irrational, $\{(2n+1)\alpha:n\in\mathbb N\}$ is
dense in $(0,1]$, thus there exists  $n\in \mathbb\N$  such that
$x,x\oplus(2n+1)\alpha\in B$. By  (\ref{LotsOfVariation}) this
implies that $|\ppsi(x\oplus(2n+1)\alpha)-\ppsi(x)|\geq 1$,
contradicting the definition of $B$.
\end{example}

We conclude with two examples which show that in general it is  not
possible to find dual maximizers
 $\pphi$ and $\ppsi$ which are
integrable even if $c$ is continuous and finitely valued.
\begin{example}\label{NoIntegrableMaximizers}
Let $X=Y=(0,1]$, take $\mu,\nu$ to be the Lebesgue measure and set
$c(x,y)=|1/x-1/y+1|$. Define $\pi$ to be the transport plan
concentrated on the diagonal such that $I_c[\pi]=1$. The functions
$\pphi(x)=1/x,\ppsi(y)=1-1/y$ witness that $\pi$ is strongly
$c$-cyclically monotone and hence optimal. Recall that by Corollary \ref{ConnectionWithStrongCmon} functions that witness strong $c$-cyclical monotonicity of a transport plan correspond precisely to dual maximizers. In particular, $\pphi$ and $\ppsi$ are dual maximizers. Of course $\pphi$ and $\ppsi $ are not integrable. Let $( \phi, \psi)$ be another pair of dual maximizers. We will see that there is a constant $\beta\in \R$  such that $\pphi= \phi+\beta, \ppsi= \psi-\beta$ almost surely,  hence there exist no integrable dual maximizers.

 Fix $a,b\in(0,1]$ such that $a<b$ and $(1/x-1/y+1)$ is positive on
$[a,b]^2$. Let $\pi_{a,b}$ be the transport plan which equals $\pi$
on $X \times Y\setminus [a,b]^2$ and is $1/(b-a)$ times the product
measure on $[a,b]^2$. As above, $\pphi$ and $\ppsi$ witness that
$\pi_{a,b}$ is strongly $c$-cyclically monotone and thus optimal. But then also $ \phi, \psi$ witness that $\pi_{a,b}$ is strongly $c$-cyclically monotone, hence
\begin{align}\label{LessOrEqual}
 \phi(x)+\psi(y)= c(x,y)=\pphi(x)+\ppsi(y)
\end{align}
for almost  all $(x,y)\in[a,b]^2$. This yields
\begin{align}
(\phi-\pphi)(x)=(\ppsi-\psi)(y),
\end{align}
thus both  sides are almost everywhere in $[a,b]$ equal up to some
constant $\beta_{a,b}$. Since we can cover $(0,1]$ with sufficiently small
overlapping  intervals $ [a_n,b_n],n\geq 1$, we achieve that this
constant $\beta_{a,b}=:\beta$ does in fact not depend on the choice of $a$ and $b$.
It follows that indeed $\phi=\pphi+\beta$ and $ \psi=\ppsi-\beta$ almost everywhere.

%
%

\end{example}

In fact, one can find an example displaying the same phenomenon as
Example \ref{NoIntegrableMaximizers} where $c$ is just the squared
distance on $\R$.

\begin{example}\label{NoIntegrableMaximizers2}
Let $X=Y=\mathbb \R$, let $c(x,y)=(x-y)^2$, let $\mu$ be a Borel probabilty measure on $X$, define $T:X\to Y$ by $T(x)=x+1$ and assume that $\nu=T\push \mu$. Let $\pi\in\Pi(\mu,\nu)$ be the obvious transport plan concentrated on the graph $\Gamma=\{(x,x+1):x\in\R\}$ of $T$. Set $\pphi(x)= -2x, \ppsi(y)=2y-1$. Then
\begin{align}
c(x,y)-\pphi(x)-\ppsi(y)= (x-y)^2-2(x-y)+1=(x-y-1)^2
\end{align}
is non negative for all $(x,y)\in X\times Y$ and $\phi(x)+\psi(y)=c(x,y)$ holds precisely for $(x,y)\in\Gamma$. Thus $\pi$ is strongly $c$-cyclically monotone and the corresponding transport costs $\int_{\Gamma} 1 \, d\pi=1$ are minimal.

We claim that the dual optimizers $\pphi,\ppsi$ are essentially unique if the Lebesgue measure $\lambda $ is absolutely continuous with respect to $\mu$.

Let $\phi,\psi$ be dual optimizers. Since $J(\phi,\psi)=\int_{X\times Y}\phi(x)+\psi(y)\,d\pi=1$, we have $\phi(x) +\psi(y)=c(x,y)=1$ for $\pi$-almost all $(x,y)\in X\times Y$.
Thus there is a set $X'\subset X$ with $\lambda(X\setminus X')=\mu(X\setminus X')=0$ such that
\begin{align}\label{OneByAnother}
\phi(x)+\psi(1+x)=c(x,x+1)=1 \ \Longrightarrow \ \psi(1+x)=1-\phi(x)
\end{align}
for all $x\in X'$. Since $\phi(x)+\psi(1+y)\leq c(x,1+y)=(x-y-1)^2$ for all $x,y\in \mathbb R$, (\ref{OneByAnother}) yields that for all $x\in X,y\in X'$
\begin{align}
\phi(x)+1-\phi(y) \ & \leq (x-y-1)^2\\
\Longrightarrow \quad \big(\phi(x)+2x\big)-\big(\phi(y)+2y\big) \ & \leq (x-y-1)^2-2(x-y-1)+1.\label{LastLine1444}
\end{align} Setting $f(x)=\phi(x)+2x$, (\ref{LastLine1444}) is tantamount to
\begin{align}
f(x)-f(y)\leq(x-y)^2
\end{align} for all $x\in X,y\in X'$. Thus $f$ is constant on $X'$, whence there exists a constant $\beta \in \mathbb R$ such that
\begin{align}\label{TheyAreAllEqual}
\phi(x)= -2x+\beta=\pphi(x), \quad \psi(y)=2y-1+\beta =\ppsi(y).
\end{align}
for $\lambda$- as well as $\mu$-almost all $x\in X$ and $\lambda$- as well as $\nu$-almost all $y\in Y$.

If we pick $\mu$ such that $\lambda $ is absolutely continuous with respect to $\mu$
and such that  $\int_\R x\,  d\mu(x)$ does not exist,  $\pphi$ and $\ppsi$ are not integrable and by (\ref{TheyAreAllEqual}) no other pair of dual maximizers can be integrable either.
\end{example}

\begin{ack}
The authors are grateful to Martin Goldstern, Gabriel Maresch and
Josef Teichmann for many helpful  discussions on the topic of this
paper.
\end{ack}

\def\ocirc#1{\ifmmode\setbox0=\hbox{$#1$}\dimen0=\ht0 \advance\dimen0
  by1pt\rlap{\hbox to\wd0{\hss\raise\dimen0
  \hbox{\hskip.2em$\scriptscriptstyle\circ$}\hss}}#1\else {\accent"17 #1}\fi}

\end{document}